\documentclass[12pt]{amsart}

\usepackage{amsfonts}
\usepackage{amsmath}
\usepackage{amssymb}
\usepackage{amstext}
\usepackage{amsthm}
\usepackage{mathtools}
\usepackage{amsopn}
\usepackage{mathrsfs}
\usepackage{tikz-cd}
\usepackage{geometry}
\usepackage{url}
\usepackage{hyperref}
\usepackage{amsrefs}
\usepackage[utf8]{inputenc}
\usepackage{bm}
\usepackage{amsopn}
\usepackage{verbatim}
\usetikzlibrary{arrows}
\usepackage [autostyle, english = american]{csquotes}
\MakeOuterQuote{"}

\hypersetup{
	colorlinks=true,
	linkcolor=red,
	filecolor=magenta,      
	urlcolor=cyan,
}

\theoremstyle{definition}
\theoremstyle{remark}
\newtheorem{lemma}{Lemma}[section]
\newtheorem{theorem}[lemma]{Theorem}
\newtheorem{proposition}[lemma]{Proposition}

\newtheorem{definition}[lemma]{Definition}
\newtheorem{remark}[lemma]{Remark}
\newtheorem{example}[lemma]{Example}
\newtheorem{notation}[lemma]{Notation}
\newtheorem{convention}[lemma]{Convention}
\newtheorem{discussion}[lemma]{Discussion}
\newtheorem{question}[lemma]{Question}
\newtheorem*{theorem1}{Theorem A}
\numberwithin{equation}{section}

\newcommand{\Z}{\mathbb{Z}}

\newcommand{\x}{\mathbf{x}}
\newcommand{\bt}{\mathbf{t}}
\newcommand{\balpha}{\bm{\alpha}}
\newcommand{\bgamma}{\bm{\gamma}}
\newcommand{\bT}{\mathbf{T}}
\newcommand{\calI}{\mathcal{I}}
\newcommand{\calB}{\mathcal{B}}
\DeclareMathOperator{\pd}{pd}

\newcommand{\lex}{\text{\footnotesize{lex}}}

\newcommand{\grevlex}{\text{\footnotesize{grevlex}}}

\DeclareMathOperator{\Borel}{Borel}

\begin{document}

\title{Equations of the multi-Rees algebra of fattened coordinate subspaces}

\author{Babak Jabbar Nezhad}
\address{Istanbul, Turkey}
\email{babak.jab@gmail.com}


\date{}

\dedicatory{}

\thanks{Babak Jabbar Nezhad has also published under the name Babak Jabarnejad~\cite{jabarnejad2016rees}}

\begin{abstract}
	In this paper we describe the equations defining the multi-Rees algebra $k[x_1,\dots,x_n][I_1^{a_1}t_1,\dots,I_r^{a_r}t_r]$, where the ideals are generated by subsets of $x_1,\dots,x_n$. We also show that a family of binomials whose leading terms are squrefree, form a Gr\"{o}bner basis for the defining equations with lexicographic order. We show that if we remove binomials that include $x$'s, then remaining binomials form a Gr\"{o}bner basis for the toric ideal associated to the multi-fiber ring. However binomials, including $x$'s, in Gr\"{o}bner basis of defining equations of the multi-Rees algebra are not necessarily defining equations of corresponding symmetric algebra. Despite this fact, we show that this family of ideals is of multi-fiber type.
\end{abstract}

\maketitle

\section{Introduction}

Let $R$ be a Noetherian ring, and let $s_1,\dots,s_n$ be generators of the ideal $I$. We define the homomorphism $\phi$ from the polynomial ring $S=R[T_1,\dots,T_n]$ to the Rees algebra $R[It]$ by sending $T_i$ to $s_it$. Then $R[It]\cong S/\ker(\phi)$. The generating set of $\ker(\phi)$ is referred to as the defining equations of the Rees algebra $R[It]$. Finding these generating sets is a tough problem which is open for most classes of ideals. Some papers about this problem are \cite{vasconcelos1991rees}, \cite{vasconcelosulrich1993rees}, \cite{morey1996rees}, \cite{moreyulrich1996rees}, \cite{kustinpoliniulrich2017blowup}.

More generally, given any ideals $I_1,\dots,I_r$ in a ring $R$, one would like to describe the equations of the multi-Rees algebra $R[I_1t_1,I_2t_2,\dots,I_rt_r]$. Indeed, the multi-Rees algebra in question is simply the Rees algebra of the module $I_1\oplus I_2\oplus\dots\oplus I_r$. However, in our work, we make no serious use of this theory. There is little work on the defining equations of the multi-Rees algebra compared to the ordinary Rees algebra. Another motivation for investigating the multi-Rees algebra is an illustration of the theory of Rees algebra of modules \cite{eisenbud2003rees}, \cite{simis2003rees}. Some works about defining equations of the multi-Rees algebra included in~\cite{ribbe1999defining}, \cite{LP14}, \cite{Sosa14}, \cite{BC17}, \cite{jabarnejad2016rees}, \cite{BC17b}, \cite{CLS19}, \cite{DJ20}.

In this paper we determine the equations of the multi-Rees algebra $R[I_1^{a_1}t_1,I_2^{a_2}t_2,\dots,I_r^{a_r}t_r]$, where $R=k[x_1,\dots,x_n]$ ($k$ a field) and ideals $I_i$ are generated by subsets of $x_1,\dots,x_n$. We present the concept of binary quasi-minors and we show that some explicit binary quasi-minors, whose leading terms are squrefree, form a Gr\"{o}bner basis with lexicographic order for defining equations. The degree of these binomials is at most $r+1$. Also, we show that if we remove binary quasi-minors that include $x$'s from Gr\"{o}bner basis, then the rest form Gr\"{o}bner basis for the toric ideal of multi-fiber ring $k[I_1^{a_1}t_1,I_2^{a_2}t_2,\dots,I_r^{a_r}t_r]$. We show that in general, if we add equations of the symmetric algebra to Gr\"{o}bner basis of toric ideal associated to the toric ring $k[I_1^{a_1}t_1,I_2^{a_2}t_2,\dots,I_r^{a_r}t_r]$, then they don't necessarily form a Gr\"{o}bner basis for the defining equations of the multi-Rees algebra. That means some of binary quasi-minors, including $x$'s, in the Gr\"{o}bner basis of defining equations of the multi-Rees algebra are not defining equations of the symmetric algebra. This is shown in Example~\ref{counter}. Note that the ideals in discussion, individually satisfy $l$-exchange property (see~\cite{HHV05}) with any monomial order. We know that in the case of Rees algebra when a monomial ideal $I$ is generated by monomials of one degree, and it satisfies $l$-exchange property, then Gr\"{o}bner basis of the defining equations of the Rees algebra $R[It]$ is formed by Gr\"{o}bner basis of defining equations of the toric ideal of toric ring $k[I]$ plus some equations of the symmetric algebra~\cite[Theorem 5.1]{HHV05}. However, when it comes to just defining equations, we show that if we add equations of symmetric algebra to defining equations of the toric ideal associated to the multi-fiber ring, then we obtain defining equations of the multi-Rees algebra. Then this family of ideals is of multi-fiber type. Notice that when powers of all ideals coincide 1, the multi-fiber ring of theses ideals is just the toric ring of edge ideals which is a well-known concept. Defining equations of the toric ideal of edge ideals is already studied in many papers including~\cite{V95}, \cite{DST95}, \cite{OHT13}, \cite{OH98}, \cite{OH99}, \cite{RTT12}, \cite{SVV98}.

We can summarize the main result of this paper as below.

\begin{theorem1}
	Let $R=k[x_1,\dots,x_n]$ and suppose that ideals $I_i$ are generated by subsets of $x_1,\dots,x_n$. Then there is a quasi-matrix $D$, whose entries are certain indeterminates, such that the multi-Rees algebra $R[I_1^{a_1}t_1,\dots,I_r^{a_r}t_r]$ is defined by the ideal generated by all binary quasi-minors of $[\underline{x}|D]$. Also, an explicit subset of these binary quasi-minors form a Gr\"{o}bner basis with lexicographic order whose leading terms are squarefree.
\end{theorem1}

As this concept could be seen as a specialization of toric ring of edge ideals, one may wonder whether we can describe defining equations of the multi-Rees algebra using graph theory. We define a bipartite graph associated to the multi-Rees algebra of these ideals and we describe defining equations using cycles of this graph. 

To prove the main result, first, we use a result in~\cite{DJ20} to show one case. Next, for more general case, we build a directed bipartite graph (this graph is different than the graph that we use to describe equations, but both graphs come from the same concept) and we find cycles of this graph and then we prove the theorem.

Actually, we can associate a bipartite graph to the multi-Rees algebra $k[x_1,\dots,x_n][I_1^{a_1}t_1,\dots,I_r^{a_r}t_r]$, discussed in this paper, as follows: one partition of vertices is formed by $t_1,\dots,t_r$. Another partition is formed by $x_1,\dots,x_n$. We attach $x_i$'s which divide generators of $I_j^{a_j}$ to $t_j$. When this graph is chordal (that means every cycle of length greater than or equal to 6 has a chord), then Gr\"{o}bner basis for defining equations of the multi-Rees algebra is the way described in~\cite{DJ20}. In the present paper we describe the Gr\"{o}bner basis for all cases, including non-chordal cases.

\begin{example}\label{ex-non-chordal}
	Consider the polynomial ring $R=k[x_1,x_2,x_3]$. Let $I_1=\langle x_1,x_2\rangle$, $I_2=\langle x_2,x_3\rangle$, $I_3=\langle x_1,x_3\rangle$. We consider the multi-Rees algebra $R[I_1^2t_1,I_2^2t_2,I_3^2t_3]$. Then incidence bipartite graph corresponding to this multi-Rees algebra is shown in Figure~\ref{non-chordal}. As we see this graph is non-chordal.
	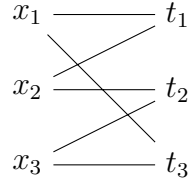
\begin{figure}
		\centering
		\begin{tikzpicture}
		\node (x1) at (0,4) {$x_1$};
		\node (x2) at (0,3) {$x_2$};
		\node (x3) at (0,2) {$x_3$};
		
		\node (t1) at (2,4) {$t_1$};
		\node (t2) at (2,3) {$t_2$};
		\node (t3) at (2,2) {$t_3$};
		
		\draw (t1)--(x2) (t1)--(x1) (t2)--(x3) (t2)--(x2) (t3)--(x1) (t3)--(x3);
		\end{tikzpicture}
		\caption{The bipartite incidence graph for Example~\ref{ex-non-chordal}}
		\label{non-chordal}	
	\end{figure}
\end{example}

Also, in~\cite{DJ20}, it is proved in more generality that if the incidence bipartite graph is chordal, then the multi-Rees algebra is Koszul. In the last section we show that in our case if the incidence bipartite graph is non-chordal, then the multi-Rees algebra is not Koszul. We also pose a question for interested readers.
\section{Background}\label{toricsection}

This section is taken from~\cite{DJ20}, as we need these results to prove our main result.

Let $G$ be a finite collection of monomials of positive degree in the polynomial ring $k[x_1,\dots,x_n]$ ($k$ is a field). We denote $\x=x_1,\dots,x_n$. Let $S$ be the polynomial ring $S:=k[T_m:m\in G]$. We define the toric map associated to $G$ as the map $\phi_G: S\to k[\x]$, where $\phi_G(T_m)=m$. We denote $J_G$ for the kernel of $\phi_G$; clearly this is a toric ideal associated to $G$. Given $\bgamma=(\gamma_m)\in\Z_{\ge 0}^{G}$, we write $\bT^{\bgamma}$ for $\prod_{m\in G}T_m^{\gamma_m}$, where $T_m$ is a variable of the polynomial ring $S=k[T_m:m\in G]$. If $\mu\in k[G]$ is a monomial, then
\[
S_\mu=\mbox{span}_k\{\bT^{\bgamma}: \phi_G(\bT^{\bgamma})=\mu\}.
\]

We now introduce a combinatorial device from~\cite{DJ20}, which we call the \textit{fiber graph} of the toric map $\phi_G$ at the monomial $\mu$.

\begin{definition}\cite{DJ20}\label{def:FiberGraph}
	Let $G$ be a finite collection of monomials of positive degree in the polynomial ring $k[\x]$, $J_G$ the toric ideal of $G$, and $\calB\subset J_G$ a finite collection of binomials. The fiber graph of $\phi_G$ at $\mu$ with respect to $\calB$ is the graph $\Gamma_{\mu,\calB}$ whose vertices are monomials $\bT^{\bgamma}\in S_\mu$ with an edge connecting $\bT^{\bgamma},\bT^{\bgamma'}\in S_\mu$ if $\bT^{\bgamma}-\bT^{\bgamma'}$ is a multiple of a binomial from $\calB$.
	
	Moreover, if $\prec$ is a monomial order on $S$, then $\vec{\Gamma}_{\mu,\calB}$ is the graph $\Gamma_{\mu,\calB}$ with edges directed from the larger monomial to the smaller. That is, if $\bT^{\bgamma},\bT^{\bgamma'}\in S_\mu$ are connected by an edge in $\Gamma_{\mu,\calB}$ and $\bT^{\bgamma'}\prec\bT^{\bgamma}$, then we get the directed edge $\bT^{\bgamma}\to \bT^{\bgamma'}$.
\end{definition}

\begin{remark}
	We suppress the collection $G$ of monomials of $k[\x]$ in the notation for $\Gamma_{\mu,\calB}$, assuming that the underlying toric map is understood from context.
\end{remark}

\begin{proposition}\cite{DJ20}\label{prop:Graph}
	Let $\phi_G:S\to k[\x]$ be a toric map and $\calB$ a collection of binomials from the toric ideal $J=J_G$. If $S$ is equipped with a monomial order $\prec$, then the following are equivalent:
	\begin{enumerate}
		\item The binomials in $\calB$ form a Gr\"obner basis for $J_G$ under $\prec$.
		\item Every nonempty graph $\vec{\Gamma}_{\mu,\calB}$ has a unique sink.
	\end{enumerate}
\end{proposition}

Let the ideal $J$ be generated by a subset of $\x$. Let $I=J^m$, $\bT=\{T_m:m\in I\}$, and $\phi:k[\bT]\to k[\x]$ be defined by $\phi(T_m)=m$. Order the variables of $k[\bT]$ by $T_{m}\succ T_{m'}$ if $m\succ_{\grevlex} m'$.  Let $\prec_\lex$ be the lexicographic monomial order on $k[\bT]$ with respect to this ordering of the variables. Set 
\begin{gather*}\label{eq:ReesAlgebraGB}
\calB=\{T_m T_n-T_{\frac{x_i}{x_j}m}T_{\frac{x_j}{x_i}n}:x_j\mid m,x_i\mid n\}.
\end{gather*}
\begin{theorem}\cite{DJ20}\label{toricidealcase}
	With assumptions above, $\calB$ is a Gr\"obner basis for $J_{k[I]}$ with respect to $\prec_\lex$.
\end{theorem}

\section{Binary quasi-minors}

\begin{definition}
	An $n\times m$ quasi-matrix over a ring $R$ is a rectangular array with $n$ rows and $m$ columns such that some entries may be empty.
	
	A subquasi-matrix is a quasi-matrix that is obtained by deleting some rows, columns, or elements of a quasi-matrix.
\end{definition}

\begin{example}
	$$
	A=\begin{bmatrix*}
	a& &b\\
	c&d& \\
	e&f&g
	\end{bmatrix*}
	$$
	is a quasi-matrix and $\begin{bmatrix*}
	a& &b\\
	&d& 
	\end{bmatrix*}$ is a subquasi-matrix of $A$.
\end{example}

\begin{definition}
	A binary quasi-matrix is a quasi-matrix having exactly two elements in each nonempty row and column.
\end{definition}

\begin{example}
	All $3\times 3$ binary quasi-matrices are listed below:
	\begin{gather*}
	\begin{bmatrix*}
	a&b&\\
	&c&d\\
	e&&f
	\end{bmatrix*},\
	\begin{bmatrix*}
	a&b&\\
	c&&d\\
	&e&f
	\end{bmatrix*},\
	\begin{bmatrix*}
	a& &b\\
	&c&d\\
	e&f&
	\end{bmatrix*},\
	\begin{bmatrix*}
	a& &b\\
	c&d& \\
	&e&f
	\end{bmatrix*},\
	\begin{bmatrix*}
	&a&b\\
	c&d& \\
	e&&f
	\end{bmatrix*},\
	\begin{bmatrix*}
	&a&b\\
	c&&d \\
	e&f&
	\end{bmatrix*}
	\end{gather*}
\end{example}

Note that a binary quasi-matrix is a square matrix, up to deleting an empty row or column. Since we usually identify a quasi-matrix canonically with the one obtained by deleting any empty row or column, in the sequel we usually consider a binary quasi-matrix as a square matrix.

\begin{definition}
	Let $A=(a_{ij})$ be an $n\times n$ binary quasi-matrix over a ring $R$. A binary quasi-determinant of $A$ is an element
	$$
	a_{1\sigma(1)}a_{2\sigma(2)}\dots a_{n\sigma(n)}-a_{1\tau(1)}a_{2\tau(2)}\dots a_{n\tau(n)}
	$$
	where $\sigma,\tau$ are permutations of $\{1,2,\dots,n\}$ such that $\sigma(l)\neq\tau(l)$ for all $1\le l\le n$. A binary quasi-determinant of a binary subquasi-matrix $A$ is called a binary quasi-minor of $A$.
\end{definition}

Note that by definition, if $\delta$ is a binary quasi-determinant of a quasi-matrix, then so is $-\delta$. In the sequel, we will usually consider a given binary quasi-minor up to sign.

\begin{remark}
	(1) Note that the quasi-determinant of a $2\times 2$ binary quasi-matrix is equal to its determinant, up to sign. Hence all $2\times 2$ minors (which exist) of a quasi-matrix are binary quasi-minors.
	
	(2) Note that a quasi-determinant of a $3\times 3$ binary quasi-matrix is uniquely determined up to sign. However, in general it is not equal to the determinant, even up to sign, of
	the matrix obtained by assigning value zero to all empty positions.
	
	(3) For $n\ge4$, a quasi-determinant of a binary $n\times n$ quasi-matrix is not even unique, up to sign. For example, consider the following binary quasi-matrix
	$$
	\begin{bmatrix*}
	a&b& & \\
	c&d& & \\
	& &e&f\\
	& &g&h
	\end{bmatrix*}.
	$$
	Then $adeh-bcgf$ and $adgf-bceh$ are both quasi-determinants.
\end{remark}

\begin{notation}
	If $A$ is a quasi-matrix with entries in $R$, then we denote the ideal generated by the binary quasi-minors of $A$ by $I_{bin}(A)$. 
\end{notation}

\begin{example}
	Consider the quasi-matrix $A$ as below:
	$$
	A=\begin{bmatrix*}
	a&b& \\
	c&d&e\\
	f& &g
	\end{bmatrix*},
	$$
	then $adg-bef$, $bef-adg$, $ad-bc$, $bc-ad$, $cg-fe$, and $fe-cg$ are all binary quasi-minors of $A$.
\end{example}

The next elementary result shows that the ideal of binary quasi-minors generalizes the classical ideal of $2\times 2$ minors.

\begin{proposition}\label{minor-binary-quasi-minor} 
	Let A be a matrix. Then $I_{bin}(A)=I_2(A)$.
\end{proposition}
\begin{proof}
	It is enough to show that every binary quasi-minor in $A$ is an $R$-combination of $2\times 2$ minors. Let $\delta=V_1V_2\dots V_n-W_1W_2\dots W_n$ be an arbitrary binary quasi-minor. We induct on $n\ge2$. Since the result is clear for $n=2$, we may assume $n\ge3$ and that the result holds for binary quasi-minors of size $<n$.
	
	We may assume $V_1$ is in the same row with $W_1$ and $V_2$ is in the same column with $W_1$. Let $U$ be the entry of $A$ in the same column as $V_1$ and same row as $V_2$. Then 
	\begin{gather*}
	\begin{aligned}
	&&\delta&=\delta-UW_1V_3\dots V_n+UW_1V_3\dots Vn\\
	&&      &=(V_1V_2-UW_1)V_3\dots V_n+W_1(UV_3\dots V_n-W_2\dots W_n).
	\end{aligned}
	\end{gather*}
	If $U$ is not one of the $W$'s, then the subquasi-matrix obtained by deleting the first row and column involving $W_1$ and $V_2$, and containing $U$ is binary quasi-matrix, with $UV_3\dots V_n-W_2\dots W_n$ as an $(n-1)$-sized binary quasi-minor.
	
	On the other hand, if $U$ is a $W_i$, say $W_2$ (which can only happen if $n\ge4$), then $UV_3\dots V_n-W_2\dots W_n=W_2(V_3\dots V_n-W_3\dots W_n)$ and $V_3\dots V_n-W_3\dots W_n$ is a
	binary quasi-minor of A of size $(n-2)$. In either case, we are done by induction.
\end{proof}

\section{Equations of the multi-Rees algebra}

Let $I_i=J_i^{a_i}$, where the $a_i$'s are positive integers and the ideals $J_i$ are generated by arbitrary subsets of $x_1,\dots,x_n$. In the rest of this paper, by generators of $J_i$ we mean these subsets of $x_1,\dots,x_n$. Also, by minimal generating set of $I_i$ we mean generators that are created from mentioned generating set of $J_i$. We denote $\boldsymbol{a}=a_1,\dots,a_r$.

Let $\calI=\{I_1,\ldots,I_r\}$. If $\balpha=(\alpha_1,\ldots,\alpha_n)\in\Z_{\ge 0}^n$, we write $\x^{\balpha}$ for $x_1^{\alpha_1}\cdots x_n^{\alpha_n}$. We write $\bt$ for the set of variables $\{t_1,\ldots,t_r\}$. We also write $k[\x,\bt]$ for the polynomial ring $k[\x][t_1,\ldots,t_r]=k[x_1,\ldots,x_n,t_1,\ldots,t_r]$. 

We consider the multi-Rees algebra of $\calI$:
\[
k[\x][\calI\bt]=k[\x][I_1t_1,\ldots,I_rt_r]=\bigoplus\limits_{b_1,\ldots,b_r\ge 0} I_1^{b_1}\cdots I_k^{b_r}t_1^{b_1}\cdots t_r^{b_r}.
\]

Let $G_1,\ldots,G_r$ be minimal sets of generators for $I_1,\ldots,I_r$. Then clearly
\[
k[\x][I_1t_1,\ldots,I_rt_k]=k[x_1,\ldots,x_n,\{\x^{\balpha} t_j: \x^{\balpha}\in G_j\}]
\]
We create a variable $T_{\x^\alpha t_j}$ for each monomial $\x^{\balpha}t_j$ and write $\bT$ for the set of all such variables. We then define the map $\phi$ as follows:
\[
\phi:S:=k[\x,\bT]\to k[\x,\bt],
\]
where $\phi(x_i)=x_i$ for all $x_i\in\x$ and $\phi(T_{\x^{\balpha}t_j})=x^{\balpha}t_j$ for all $\x^{\balpha}t_j$ with $\x^{\balpha}\in G_j$.  Clearly this is a toric map as discussed in Section~\ref{toricsection}. We are concerned primarily with the defining equations of $k[\x][I_1t_1,\dots,I_rt_r]$, which is the toric ideal $\ker(\phi)$.

Let $\mathfrak{m}$ for the ideal $\langle x_1,\ldots,x_n\rangle\subset k[\x]$. Then the multi-fiber ring of $k[\x][\calI\bt]$ is $k[\x][\calI\bt]/\mathfrak{m}k[x][\calI\bt]$. Since in our case the monomial ideals $\calI=\{I_1,\ldots,I_r\}$ are each generated in a single degree, then we have an isomorphism
\[
k[\x][\calI\bt]/\mathfrak{m}k[\x][\calI\bt]\cong k[\x^{\balpha} t_j: \x^{\balpha}\in G_j, 1\le j\le r].
\]
We denote the ring $k[\x^{\balpha} t_j: \x^{\balpha}\in G_j, 1\le j\le r]$ by $k[\calI\bt]$. We define the map $\psi$ as follows:
\[
\psi: k[\bT]\to k[\x,\bt],
\]
where $\psi(T_{\x^{\balpha}t_j})=x^{\balpha}t_j$.

If $\calI=\{I\}$ consists of a monomial ideal generated in a single degree, then
\[
k[\x][It]/\mathfrak{m}k[\x][It]\cong k[\x][It]\cong k[I],
\]
is the toric ring of $I$.

The following definition can be found in~\cite{BC17}.

\begin{definition}
	The ideal $I$ that is generated by the elements $f_1,\dots,f_m$ of single degree is said to be of fiber type if the defining ideal of the Rees algebra $k[\x][It]$ is generated by polynomials that either (1) are linear in the indeterminates representing the generators $f_it$ of the Rees algebra (and therefore are relations of the symmetric algebra), or (2) belong to the defining ideal of $k[I]$. One can immediately generalize this notion and speak of family $I_1,\dots,I_r$ of multi-fiber type.
\end{definition}

We will see that the family of ideals $I_1,\dots,I_r$ in our case is of multi-fiber type. 

\begin{remark}
	In~\cite[Definition 4.1]{HHV05}, authors define $l$-exchange property for a monomial ideal, whose generators have one degree. They also show in \cite[Theorem 5.1]{HHV05} that if an ideal $I$ satisfies $l$-exchange property, then Gr\"{o}bner basis of defining equations of the Rees algebra $k[\x][It]$ is formed by Gr\"{o}bner basis of defining equations of $k[I]$ plus some of relations of symmetric algebra. We see easily that in our case each of ideals $I_1,\dots,I_r$ satisfies $l$-exchange property with any monomial order. We will see in the Example~\ref{counter}, that even under the conditions that mentioned in the~\cite[Theorem 5.1]{HHV05}, this theorem does not hold in our case. That means the Gr\"{o}bner basis will have some polynomials involving $x$'s which are not symmetric polynomials. Even though, as we mentioned, these family of ideals is of multi-fiber type. 
\end{remark}

\begin{convention}[Monomial order for $k\left\lbrack\x,\bT\right\rbrack$]\label{conv:MonomialOrder}
	We fix the following monomial order on $k[\x,\bT]$. We order the variables of $k[\x,\bT]$ by $T_{mt_i}\succ T_{nt_j}$ if $i>j$ or $i=j$ and $m\succ_{\grevlex} n$. Furthermore we make $T_{mt_i}\succ x_j$ for any variable $T_{mt_i}\in\bT$ and $x_j\in\x$.  On top of this ordering of the variables of $k[\x,\bT]$ we put the lexicographic order.
\end{convention}

\begin{definition}
	For a fixed $l$, we define the matrix $B_{a_l}$, whose entries in first row are $T_{x_1^{p_1}\dots x_n^{p_n}t_l}$, where $x_1^{p_1}\dots x_n^{p_n}$ are generators of $\mathfrak{m}^{a_l}$, $p_1\ge 1$, and smaller elements are put on the left. For each $T_{x_1^{p_1}\dots x_n^{p_n}t_l}$ on the first row entry under this element on the $v$-th row is 
	$$
	T_{x_1^{p_1-1}\dots x_v^{p_v+1}\dots x_n^{p_n}t_l}=T_{\frac{x_v}{x_1}x_1^{p_1}\dots x_v^{p_v}\dots x_n^{p_n}t_l},\ 1\le v\le n.
	$$
\end{definition} 

We see that if $T_{mt_l}$ is in the column $v$ and row $i$ of $B_{a_l}$, then the entry in the same column and row $j$ is $T_{\frac{x_j}{x_i}mt_l}$. Also, we see that if all distinct factors of the monomial $m$ of degree $a_l$ are $x_{i_1},\dots,x_{i_v}$, then $T_{mt_l}$ appears in the rows $i_1,\dots,i_v$ of $B_{a_l}$.

\begin{definition}
	Let $J_l=\langle x_{i_1},\dots,x_{i_p}\rangle$. We define the quasi-matrix $D_{a_l}$ to be the subquasi-matrix of $B_{a_l}$ by choosing $T_{m t_l}$ from $i_1$-row such that $m$ are monomials in $x_{i_1},\dots,x_{i_p}$. Then under these elements we choose entries on rows $i_2,\dots,i_p$. This is nothing but choosing all $T_{m t_l}$ such that $m$ are monomials in $x_{i_1},\dots,x_{i_p}$.
	
	We define the subquasi-matrices $D_{\boldsymbol{a}}\coloneqq(D_{a_1}|D_{a_2}|\dots|D_{a_r})$ and $C_{\boldsymbol{a}}\coloneqq(\mathbf{x}|D_{\boldsymbol{a}})$.
\end{definition}

\begin{lemma}
	$I_{bin}(C_{\boldsymbol{a}})\subseteq\ker(\phi)$.
\end{lemma}
\begin{proof}
	Let $\alpha$ be a binary quasi-minor of $C_{\boldsymbol{a}}$. It is enough to show that $\phi(\alpha)=0$. We prove the case that binary quasi-minor involves $x$, where the proof for the other case is similar. Suppose $i_1,\dots,i_v$ are the rows that factors of $\alpha$ appear. Without loss of generality we may assume that $x$'s appear in the rows $i_1$ and $i_v$. Let one term of $\alpha$ be $$
	x_{i_1}T_{m_1j_{l_1}}T_{m_2j_{l_2}}\dots T_{m_{v-1}j_{l_{v-1}}},
	$$ 
	where $T_{m_s{l_s}}$ is in the row $i_{s+1}$, and $l_s$ are not necessarily distinct. Then without loss of generality we may assume that other term of $\alpha$ is
	$$
	x_{i_v}T_{\frac{x_{i_1}}{x_{i_2}}m_1j_{l_1}}T_{\frac{x_{i_2}}{x_{i_3}}m_2j_{l_2}}\dots T_{\frac{x_{i_{v-1}}}{x_{i_v}}m_{v-1}j_{l_{v-1}}}.
	$$
	Then clearly $\phi(\alpha)=0$. This completes the proof.   
\end{proof}

\begin{theorem}\label{main-result}
	Suppose that ideals $J_i$ are generated by subsets of  $x_1,\dots,x_n$. Let $I_i=J_i^{a_i}$. Suppose $\calI=\{I_1,\dots,I_r\}$. Then the multi-Rees algebra $k[\x][\calI\bt]$ is normal and Cohen-Macaulay. Explicitly, a Gr\"obner basis of binary quasi-minors of $C_{\boldsymbol{a}}$ whose leading terms are squarefree is given for $\ker(\phi)$, with lexicographic order as follows: (1) Some of them are $2\times 2$-minors of $(\mathbf{x}|D_{a_l})$. (2) Some of them are binary quasi-minors such that in each term we have at most one entry from $\mathbf{x}$ and each $D_{a_l}$. 
\end{theorem}
\begin{proof}
	We use Proposition~\ref{prop:Graph}, with $\calB$ the indicated set. Under the given monomial order, we seek to show that the directed graphs $\vec{\Gamma}_{\mu,\calB}$ have a unique sink. Let $M',N'\in S_\mu$. 
	
	Let $M=\frac{M'}{\gcd(M',N')}$ and $N=\frac{N'}{\gcd(M',N')}$. Hence $\phi(M)=\phi(N)$, $M,N\in S_\nu$. Therefore to prove the theorem it is enough to show that $M$ or $N$ is not a sink, and to do this we build the directed graph $\vec{\Theta}_{M}$ (resp. $\vec{\Theta}_{N}$) associated to $M$ (resp. $N$). $\vec{\Theta}_{M}$ and $\vec{\Theta}_{N}$ are bipartite graphs with the same vertices. One partition for both of them is a subset of $x$'s and other partition is a subset of $t_j$'s ($0\le j\le r$). Note that here we create an extra variable $t_0$.
	
	For each $1\le j\le r$ (not necessarily every such $j$ will be considered), we consider $M_j$ (resp. $N_j$) to be the product of all factors of $M$ (resp. $N$) involving $t_j$. If $\phi(M_j)=\phi(N_j)$, then by Theorem~\ref{toricidealcase} and Proposition~\ref{prop:Graph}, we see that either $M_j$ or $N_j$ is not a sink. So that $M$ or $N$ is not a sink. Then we assume that for every such a $j$, $\phi(M_j)\neq \phi(N_j)$. Since number of factors involving $t_j$ in both $M_j$ and $N_j$ is equal, there is $x_\alpha$ (resp. $x_\beta$) ($x_\alpha\neq x_\beta$) in $\phi(M_j)$ (resp. $\phi(N_j)$) such that its power is greater in $\phi(M_j)$ (resp. $\phi(N_j)$). Hence there is $T_{mt_j}$ in $M_j$ where $x_\alpha$ divides $m$. Also, there is $T_{m't_j}$ in $N_j$, where $x_\beta$ divides $m'$. 
	
	Back to building the graphs, $x_\alpha$, $x_\beta$, and $t_j$ are vertices of both graphs. In the graph $\vec{\Theta}_{M}$ (resp. $\vec{\Theta}_{N}$) one directed edge goes from $x_\alpha$ to $t_j$ (resp. from $t_j$ to $x_\alpha$) and other  directed edge goes from $t_j$ to $x_\beta$ (resp. from $x_\beta$ to $t_j$).
	
	On the other hand, since $\phi(M)=\phi(N)$, there is a factor of $x_\beta$ (resp. $x_\alpha$) in $\frac{\phi(M)}{\phi(M_j)}$ (resp. $\frac{\phi(N)}{\phi(N_j)}$). Then either $x_\beta$ divides $M$ or there is $T_{nt_{j'}}$ in $M$, where $x_\beta$ divides $n$. In the former case, in the graph $\vec{\Theta}_{M}$ (resp. $\vec{\Theta}_{N}$) one directed edge goes from $x_\beta$ to $t_0$ (resp. one directed edge goes from  $t_0$ to $x_\beta$) . In the latter case, in the graph $\vec{\Theta}_{M}$ (resp. $\vec{\Theta}_{N}$) one directed edge goes from $x_\beta$ to $t_{j'}$ (resp. one directed edge goes from $t_{j'}$ to $x_\beta$). Also, either $x_\alpha$ divides $N$ or there is $T_{n't_{j''}}$ in $N$, where $x_\alpha$ divides $n'$ ($j'$ and $j''$ are not necessarily different). If $x_\alpha$ divides $N$, then in the graph $\vec{\Theta}_{M}$ (resp. $\vec{\Theta}_{N}$) one directed edge goes from $t_0$ to $x_\alpha$ (resp. one directed edge goes from $x_\alpha$ to $t_0$). If  there is $T_{n't_{j''}}$ in $N$, where $x_\alpha$ divides $n'$, then in the graph $\vec{\Theta}_{M}$ (resp. $\vec{\Theta}_{N}$) one directed edge goes from $t_{j''}$ to $x_\alpha$ (resp. one directed edge goes from $x_\alpha$ to $t_{j''}$).
	
	Finally, if $m$ and $n$ are $x$-monomials of $M$ and $N$ (in the case that they exist, and if one of them exists, then other one also exists), then $\gcd(m,n)=1$. But total degree of $m$ and $n$ is the same. Therefore, there are $x_\gamma$ and $x_\delta$ ($x_\gamma\neq x_\delta$) such that $x_\gamma\mid m$, $x_\gamma\nmid n$, $x_\delta\mid n$, $x_\delta\nmid m$. Hence there is $T_{pt_{j'}}$ in $M$ (resp. $T_{qt_{j''}}$ in $N$) ($j'$ and $j''$ are not necessarily different) such that $x_\delta$ divides $p$ (resp. $x_\gamma$ divides $q$). In the graph $\vec{\Theta}_{M}$ (resp. $\vec{\Theta}_{N}$) one directed edge goes from $x_\gamma$ to $t_0$ (resp. from $t_0$ to $x_\gamma$) and another directed edge goes from $t_0$ to $x_\delta$ (resp. from $x_\delta$ to $t_0$). Also, in the graph $\vec{\Theta}_{M}$ (resp. $\vec{\Theta}_{N}$) one directed edge goes from $x_\delta$ to $t_{j'}$ (resp. from $t_{j'}$ to $x_\delta$) and another directed edge goes from $t_{j''}$ to $x_\gamma$ (resp. from $x_\gamma$ to $t_{j''}$). 
	
	The point of building the graph $\vec{\Theta}_{M}$ (similar for $\vec{\Theta}_{N}$) is that if in this graph $x_i$ goes to $t_j$ ($j\neq 0$), then there is a $T_{mt_j}$ dividing $M$ such that $x_i$ divides $m$. If $x_i$ goes to $t_0$, then $x_i$ divides $M$. Now, if $t_j$ ($j\neq 0$) goes to $x_i$, then $x_i$ is a generator of $J_j$. Finally if $t_0$ goes to $x_i$, then it means $x_i$ is between generators of at least one of $J_j$.
	
	As we see direction of edges in the graph $\vec{\Theta}_{M}$ is opposite to direction of edges in the graph $\vec{\Theta}_{N}$. In the graph $\vec{\Theta}_{M}$ in-degree and out-degree of all vertices are greater or equal to 1 (the same in $\vec{\Theta}_{N}$). Thus there is a cycle in $\vec{\Theta}_{M}$ which we denote by $C_M$. If we change the direction of edges in $C_M$, then we obtain a cycle $C_N$ in $\vec{\Theta}_{N}$. 
	
	Let $j_l$ be the maximum index between $j$, where $t_j$ are vertices of one partition in $C_M$. Let $x_{i_l}, x_{i_{l-1}}$ be in $C_M$ such that $x_{i_l}$ goes to $t_{j_l}$ and $x_{i_{l-1}}$ leaves $t_{j_l}$. Also, let $t_{j_1},\dots,t_{j_l}$ be vertices of one partition in $C_M$ and $x_{i_1},\dots,x_{i_l}$ be vertices in other partition in $C_M$. The cycle $C_M$ is shown in Figure~\ref{directedgraph}.
	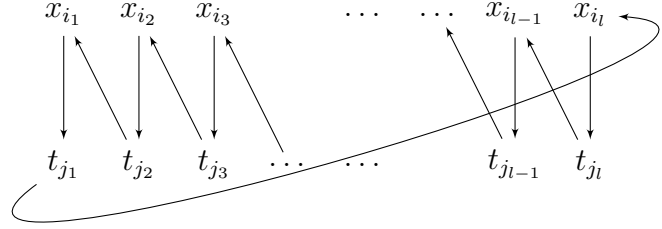
\begin{figure}
		\centering
		\begin{tikzpicture}
		\tikzset{vertex/.style = }
		\tikzset{edge/.style = {->,> = latex'}}
		\node[vertex] (xi1)    at  (0,2) {$x_{i_1}$};
		\node[vertex] (xi2)    at  (1,2) {$x_{i_2}$};
		\node[vertex] (xi3)    at  (2,2) {$x_{i_3}$};
		\node[vertex] ()      at  (4,2) {$\dots$};
		\node[vertex] (a)      at  (5,2) {$\dots$};
		\node[vertex] (xil-1)  at  (6,2) {$x_{i_{l-1}}$};
		\node[vertex] (xil)    at  (7,2) {$x_{i_l}$};
		
		\node[vertex] (tj1)    at  (0,0) {$t_{j_1}$};
		\node[vertex] (tj2)    at  (1,0) {$t_{j_2}$};
		\node[vertex] (tj3)    at  (2,0) {$t_{j_3}$};
		\node[vertex] (b)      at  (3,0) {$\dots$};
		\node[vertex] ()      at  (4,0) {$\dots$};
		\node[vertex] (tjl-1)  at  (6,0) {$t_{j_{l-1}}$};
		\node[vertex] (tjl)    at  (7,0) {$t_{j_l}$};
		\draw[edge] (b)   to (xi3);
		\draw[edge] (tjl-1)   to (a);
		\draw[edge] (xil)   to (tjl);
		\draw[edge] (tjl)   to (xil-1);
		\draw[edge] (xil-1) to (tjl-1);
		\draw[edge] (xi3)   to (tj3);
		\draw[edge] (tj3)   to (xi2);
		\draw[edge] (xi2)   to (tj2);
		\draw[edge] (tj2)   to (xi1);
		\draw[edge] (xi1)   to (tj1);
		\draw[edge][bend right=160] (tj1)   to (xil);
		\end{tikzpicture}
		\caption{The directed cycle $C_M$ in the Theorem~\ref{main-result}}
		\label{directedgraph}
	\end{figure}
	If $x_{i_l}\succ_\lex x_{i_{l-1}}$, then by considering properties of the graph $\vec{\Theta}_{M}$, we have
	\begin{gather*}
	M''=\frac{T_{\frac{x_{i_l}}{x_{i_1}}m_1t_{j_1}}}{T_{m_1}t_{j_1}}\dots\frac{T_{\frac{x_{i_{l-2}}}{x_{i_{l-1}}}m_{l-1}t_{j_{l-1}}}}{T_{m_{l-1}}t_{j_{l-1}}}\frac{T_{\frac{x_{i_{l-1}}}{x_{i_l}}m_lt_{j_l}}}{T_{m_l}t_{j_l}}M,
	\end{gather*}
	just we should remark that if some $j_u=j_0$ ($1\le u\le l$), then in above formula, instead of $T_{m_ut_{j_u}}$ we will have $x_{i_u}$ and instead of $T_{\frac{x_{i_{u-1}}}{x_{i_u}}m_ut_{j_u}}$ we will have $x_{i_{u-1}}$. Then $M\succ_\lex M''$ and $M\to M''$, so that $M$ is not a sink.
	
	If $x_{i_{l-1}}\succ_\lex x_{i_l}$, then by a similar argument we can show that $N$ is not a sink. Note that 
	$$
	T_{m_l}t_{j_l}T_{m_{l-1}}t_{j_{l-1}}\dots T_{m_1}t_{j_1}-T_{\frac{x_{i_{l-1}}}{x_{i_l}}m_lt_{j_l}}T_{\frac{x_{i_{l-2}}}{x_{i_{l-1}}}m_{l-1}t_{j_{l-1}}}\dots T_{\frac{x_{i_l}}{x_{i_1}}m_1t_{j_1}},
	$$
	is a binary quasi-minor, as $T_{m_{l-s}}t_{j_{l-s}}$ is in the same row with $T_{\frac{x_{i_{l-s}}}{x_{i_{l-s+1}}}m_{l-s+1}t_{j_{l-s+1}}}$ ($1\le s\le l-1$), and $T_{m_l}t_{j_l}$ is in the same row with $T_{\frac{x_{i_l}}{x_{i_1}}m_1t_{j_1}}$. Also, $T$ variables involving $t_{j_{l-s}}$ ($0\le s\le l-1$) are in the same columns.
	
	Finally, we conclude normality by~\cite[Proposition 13.5, Proposition 13.15]{sturmfels1996grobner}. Cohen-Macaulayness is concluded from~\cite[Proposition 1, Theorem 1]{hochster1972rings}.
\end{proof}

We can prove the following result similar to Theorem~\ref{main-result}.

\begin{theorem}
	Suppose that ideals $J_i$ are generated by subsets of  $x_1,\dots,x_n$. Let $I_i=J_i^{a_i}$. Suppose $\calI=\{I_1,\dots,I_r\}$. Then the multi-fiber ring $k[\calI\bt]$ is normal and Cohen-Macaulay. Explicitly, a Gr\"obner basis of binary quasi-minors of $D_{\boldsymbol{a}}$ whose leading terms are squarefree is given for $\ker(\psi)$, as follows: (1) Some of them are $2\times 2$-minors of $D_{a_l}$. (2) Some of them are binary quasi-minors such that in each term we have at most one entry from each $D_{a_l}$. We take the monomial order on $k[\bT]$ to be the monomial order induced on $k[\bT]$ as a subring of $k[\x,\bT]$, where the latter is given the monomial order of Convention~\ref{conv:MonomialOrder}. 
\end{theorem}

\begin{proposition}
	Suppose that ideals $J_i$ are generated by subsets of  $x_1,\dots,x_n$. Let $I_i$ be powers of $J_i$. Then the family of ideals $I_1,\dots,I_r$ is of multi-fiber type.
\end{proposition}
\begin{proof}
	We prove that every $x$-binary quasi-minor of $C_{\boldsymbol{a}}$ (binary quasi-minors that involve $x$) can be generated by $2\times 2$ $x$-minors and $T$-binary quasi-minors of $C_{\boldsymbol{a}}$. Let $f=x_iV_1V_2\dots V_m-x_jW_1W_2\dots W_m$ ($V_i$ and $W_i$ are $T$ variables). Without loss of generality we may assume that $x_i$ and $W_1$ are in the same row and $W_1$ and $V_1$ are in the same column. If $V_1$ and $x_j$ are in the same row, then we have 
	\begin{gather*}
	f=x_iV_1V_2\dots V_m-x_jW_1W_2\dots W_m-x_jW_1V_2\dots V_m+x_jW_1V_2\dots V_m\\
	=(x_iV_1-x_jW_1)V_2\dots V_m+x_jW_1(V_2\dots V_m-W_2\dots W_m).
	\end{gather*}
	
	If $V_1$ and $x_j$ are not in the same row, then there is an $x_v$ which is in the same row with $V_1$. We have
	\begin{gather*}
	f=x_iV_1V_2\dots V_m-x_jW_1W_2\dots W_m-x_vW_1V_2\dots V_n+x_vW_1V_2\dots V_m\\
	=(x_iV_1-x_vW_1)V_2\dots V_m+W_1(x_vV_2\dots V_m-x_jW_2\dots W_m).
	\end{gather*}
	We can continue this procedure until all generators are either $2\times 2$ $x$-minors or $T$-binary quasi-minors of $C_{\boldsymbol{a}}$.   
\end{proof}

\begin{remark}
	Instead of lexicographic order we can put graded reverse lexicographic order on $k[\x,\bT]$, and in much simpler way we can prove that binary quasi-minors form a Gr\"{o}bner basis for $\ker(\phi)$ (resp. $\ker(\psi)$). But this comes at the cost of leading terms which are no longer squarefree.
\end{remark}

\begin{example}\label{counter}
	We consider the polynomial ring $R=k[x_1,x_2,x_3]$. Let $I_1=\langle x_1^2,x_1x_2,x_2^2\rangle$, $I_2=\langle x_2^2,x_2x_3,x_3^2\rangle$, $I_3=\langle x_1^2,x_1x_3,x_3^2\rangle$. We consider the multi-Rees algebra $R[I_1t_1,I_2t_2,I_3t_3]$. We have
	$$
	C_{\boldsymbol{a}}=\begin{bmatrix}
	x_1&T_{x_1x_2t_1}&T_{x_1^2t_1} &             &             &T_{x_1x_3t_3}&T_{x_1^2t_3} \\
	x_2&T_{x_2^2t_1} &T_{x_1x_2t_1}&T_{x_2x_3t_2}&T_{x_2^2t_2} &             &             \\
	x_3&             &             &T_{x_3^2t_2} &T_{x_2x_3t_2}&T_{x_3^2t_3} &T_{x_1x_3t_3}\\
	\end{bmatrix}.	
	$$
	So that $\ker(\phi)$ is generated by $T$-binary quasi-minors and $2\times 2$ $x$-minors of $C_{\boldsymbol{a}}$. But the reduced Gr\"{o}bner basis has some $x$-binary quasi-minors that are not $2\times 2$-minors. For example 
	$$
	x_2T_{x_2x_3t_2}T_{x_1^2t_3}-x_1T_{x_2^2t_2}T_{x_1x_3t_3},
	$$
	is an $x$-binary quasi-minor. Its leading term is $x_2T_{x_2x_3t_2}T_{x_1^2t_3}$, which is not divisible by leading term of any $2\times 2$-minor. 
	
	We see that if we keep everything in the order described in Convention~\ref{conv:MonomialOrder}, except we put the order $x_i\succ T_{mt_j}$, then again with a similar proof we see binary quasi-minors form a Gr\"{o}bner basis for $\ker(\phi)$ (resp. $\ker(\psi)$). Then in this case all conditions of~\cite[Theorem 5.1]{HHV05}, are satisfied. However, under this order, the leading term of binary quasi-minor
	$$
	x_2T_{x_1^2t_1}T_{x_3^2t_3}-x_3T_{x_1x_2t_1}T_{x_1x_3t_3},
	$$ 
	which is $x_2T_{x_1^2t_1}T_{x_3^2t_3}$, is not divisible by leading term of any $2\times 2$-minor. Then \cite[Theorem 5.1]{HHV05}, does not hold. Even though, we know that this family of ideals is of multi-fiber type.
\end{example}

\begin{example}
	Let $R=k[x_1,x_2,x_3]$. Let $I_1=\langle x_1,x_2\rangle$, $I_2=\langle x_2,x_3\rangle$, $I_3=\langle x_1,x_3\rangle$. Then $\ker(\phi)$ is
	$$
	\langle x_1T_{x_2t_1}-x_2T_{x_1t_1},x_1T_{x_3t_3}-x_3T_{x_1t_3},x_2T_{x_3t_2}-x_3T_{x_2t_2},T_{x_1t_1}T_{x_2t_2}T_{x_3t_3}-T_{x_1t_3}T_{x_2t_1}T_{x_3t_2}\rangle.
	$$
	We see that $C_{\boldsymbol{a}}$ has the form below
	$$
	\begin{bmatrix*}
	x_1&T_{x_1t_1}&   &T_{x_1t_3}\\
	x_2&T_{x_2t_1}&T_{x_2t_2}&   \\
	x_3&   &T_{x_3t_2}&T_{x_3t_3}
	\end{bmatrix*}.
	$$
	
	This special example can also be recovered by using the theory of Rees algebras of modules, as follows. The module $M=I_1\oplus I_2\oplus I_3$ has a linear resolution
	$$
	0\rightarrow R^3\xrightarrow{\Phi}R^6\rightarrow M\rightarrow 0
	$$
	where
	$$
	\Phi=\begin{bmatrix*}
	x_2 &0   &0   \\
	-x_1&0   &0   \\
	0   &x_3 &0   \\
	0   &-x_2&0   \\
	0   &0   &x_3 \\
	0   &0   &-x_1
	\end{bmatrix*}.
	$$
	Hence $\pd{M}=1$. Furthermore, since $M$ is free in codimension 1, and 4-generated in codimension 2, by \cite[Proposition 4.11]{simis2003rees} the Rees algebra of $M$, which is the multi-Rees algebra in question, has the expected defining equations, in the sense that
	$$
	\mathcal{R}(M)\cong R[T_1,T_2,T_3,T_4,T_5,T_6]/\langle[x_1 x_2 x_3]B,\det{B}\rangle
	$$
	where
	$$
	B=\begin{bmatrix*}
	-T_2&0   &-T_6\\
	T_1 &-T_4&0   \\
	0   &T_3 &T_5
	\end{bmatrix*}
	$$
	is the matrix defined by the equation
	$$
	[T]\Phi=[x]B.
	$$
	In this special example equations of multi-fiber ring are also known by~\cite[Proposition 3.1]{V95}. Moreover, in this example binary quasi-minors form a universal Gr\"{o}bner basis (c.f.\cite[Corllary 5.12]{HHO18}).
\end{example}

\begin{example}\label{ex-graph}
	Let $R=k[x_1,x_2,x_3,x_4]$. Let $I_1=\langle x_1^2,x_1x_2,x_2^2\rangle$, $I_2=\langle x_1^3,x_1^2x_3,x_1x_3^2,x_3^3\rangle$, $I_3=\langle x_2^2,x_2x_3,x_3^2\rangle$, $I_4=\langle x_1,x_4\rangle$, $I_5=\langle x_2,x_4\rangle$. We see that $C_{\boldsymbol{a}}$ has the form below
	$$
	\begin{bmatrix*}
	x_1&T_{x_1x_2t_1}&T_{x_1^2t_1}&T_{x_1x_3^2t_2}&T_{x_1^2x_3t_2}&T_{x_1^3t_2}   &             &             &T_{x_1t_4}&          \\
	x_2&T_{x_2^2t_1} &T_{x_1x_2t_1}&               &               &               &T_{x_2x_3t_3}&T_{x_2^2t_3} &          &T_{x_2t_5}\\
	x_3&             &            &T_{x_3^3t_2}   &T_{x_1x_3^2t_2}&T_{x_1^2x_3t_2}&T_{x_3^2t_3} &T_{x_2x_3t_3}&          &          \\
	x_4&             &            &               &               &               &             &             &T_{x_4t_4}&T_{x_4t_5}
	\end{bmatrix*}.
	$$
	For example 
	$$
	T_{x_1x_2t_1}T_{x_2^2t_3}T_{x_1x_3^2t_2}-T_{x_2^2t_1}T_{x_2x_3t_3}T_{x_1^2x_3t_2}
	$$
	is one of generators. We cannot apply \cite[Proposition 4.11]{simis2003rees} and \cite[Proposition 3.1]{V95} on this example.
	
	As we said in the introduction, when powers of ideals is 1, the multi-fiber ring of these ideals is actually toric ring of edge ideals. This brings the idea to describe binary quasi-minors as cycles of a bipartite graph, even if when powers of ideals are not necessarily 1. We show it in this example. We will also argue this generally in Discussion~\ref{cycle}. Now, we create a bipartite graph. One partition of vertices are $x$'s and other partition are 1 and $T_{mt_j}$, where $m$ are generators of $J_j^{a_j-1}$ ($I_j=J_j^{a_j}$). We connect all $x$'s to 1. We also connect all $x$'s that are generators of $J_j$ to $T_{mt_j}$. In this example our graph will be as Figure~\ref{fig:bipartiteincidencegraph}. 
	\begin{figure}
		\centering
		\begin{tikzpicture}
		\node (x1) at (0,7) {$x_1$};
		\node (x2) at (0,6) {$x_2$};
		\node (x3) at (0,5) {$x_3$};
		\node (x4) at (0,4) {$x_4$};
		
		\node (1)       at (10,10)  {$1$};
		\node (Tx2t1)   at (10,9)   {$T_{x_2t_1}$};
		\node (Tx1t1)   at (10,8)   {$T_{x_1t_1}$};
		\node (Tx32t2)  at (10,7)   {$T_{x_3^2t_2}$};
		\node (Tx1x3t2) at (10,6)   {$T_{x_1x_3t_2}$};
		\node (Tx12t2)  at (10,5)   {$T_{x_1^2t_2}$};
		\node (Tx3t3)   at (10,4)   {$T_{x_3t_3}$};
		\node (Tx2t3)   at (10,3)   {$T_{x_2t_3}$};
		\node (Tt4)     at (10,2)   {$T_{t_4}$};
		\node (Tt5)     at (10,1)   {$T_{t_5}$};
		
		\draw (1)--(x1) (1)--(x2) (1)--(x3) (1)--(x4) (Tx2t1)--(x1) (Tx2t1)--(x2) (Tx1t1)--(x1) (Tx1t1)--(x2) (Tx32t2)--(x1) (Tx32t2)--(x3) (Tx1x3t2)--(x1) (Tx1x3t2)--(x3) (Tx12t2)--(x1) (Tx12t2)--(x3) (Tx3t3)--(x2) (Tx3t3)--(x3) (Tx2t3)--(x2) (Tx2t3)--(x3) (Tt4)--(x1) (Tt4)--(x4) (Tt5)--(x2) (Tt5)--(x4);
		\end{tikzpicture}
		\caption{The bipartite graph for Example~\ref{ex-graph}}
		\label{fig:bipartiteincidencegraph}
	\end{figure}
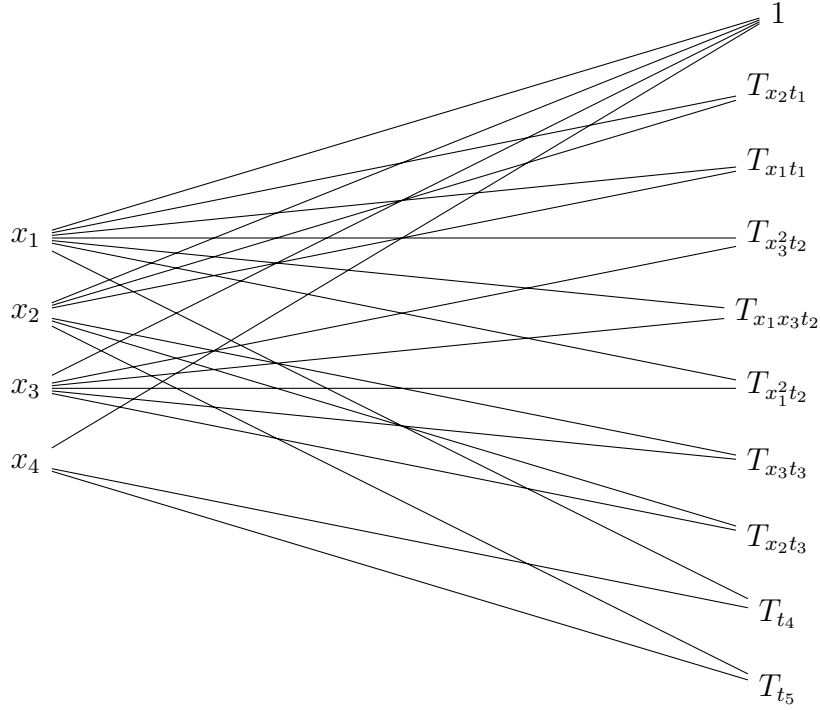
	Now every cycle in this graph gives us a binary quasi-minor. If we label edges with numbers in order, then for every even edge, where one vertex of this edge is $x_l$, two cases occur: if the other vertex is 1, then we get a factor $x_l$ in one term of binary quasi-minor. If the other vertex is $T_{mt_j}$, then we get a factor $T_{x_lmt_j}$ for the mentioned term. We do the same thing with odd edges to get factors of other term of binary quasi-minor. For example, the binary quasi-minor 
	$$
	T_{x_1x_2t_1}T_{x_2^2t_3}T_{x_1x_3^2t_2}-T_{x_2^2t_1}T_{x_2x_3t_3}T_{x_1^2x_3t_2}
	$$
	is corresponding to the cycle given in Figure~\ref{cycle-1}.
	\begin{figure}
		\centering
		\begin{tikzpicture}
		\node (x1) at (0,3) {$x_1$};
		\node (x2) at (0,2) {$x_2$};
		\node (x3) at (0,1) {$x_3$};

		\node (Tx2t1)   at (3,3)   {$T_{x_2t_1}$};
		\node (Tx2t3)   at (3,2)   {$T_{x_2t_3}$};
		\node (Tx1x3t2) at (3,1)   {$T_{x_1x_3t_2}$};
		
		\draw (Tx2t1)--(x1) (Tx2t1)--(x2) (Tx2t3)--(x2) (Tx2t3)--(x3) (Tx1x3t2)--(x1) (Tx1x3t2)--(x3);   
		\end{tikzpicture}
		\caption{Corresponding cycle to $T_{x_1x_2t_1}T_{x_2^2t_3}T_{x_1x_3^2t_2}-T_{x_2^2t_1}T_{x_2x_3t_3}T_{x_1^2x_3t_2}$}
		\label{cycle-1}
	\end{figure}
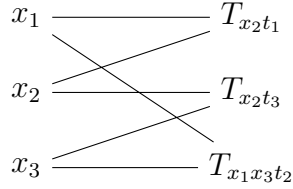
	Another easy example is $2\times 2$-minor $x_1T_{x_4t_4}-x_4T_{x_1t_4}$. The corresponding cycle is shown in Figure~\ref{cycle-2}.
	\begin{figure}
		\centering
		\begin{tikzpicture}
		\node (x1) at (0,3) {$x_1$};
		\node (x4) at (0,1) {$x_4$};

		\node (1)   at (3,3)   {$1$};
		\node (Tt4) at (3,1)   {$T_{t_4}$};
		
		\draw (1)--(x1) (1)--(x4) (Tt4)--(x1) (Tt4)--(x4) ;   
		\end{tikzpicture}
		\caption{Corresponding cycle to $x_1T_{x_4t_4}-x_4T_{x_1t_4}$}
		\label{cycle-2}
	\end{figure}
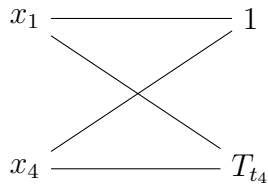
\end{example}

\begin{discussion}\label{cycle}
	Let $\alpha$ be a binary quasi-minor of $C_{\boldsymbol{a}}$. We claim that $\alpha$ is corresponding to a cycle as we mentioned in Example~\ref{ex-graph}. We prove the claim for the case that binary quasi-minor involves $x$. The other case is similar. Suppose $i_1,\dots,i_v$ are the rows that factors of $\alpha$ appear. Without loss of generality we may assume that $x$'s appear in the rows $i_1$ and $i_v$. Let one term of $\alpha$ be 
	$$
	x_{i_1}T_{m_1j_{l_1}}T_{m_2j_{l_2}}\dots T_{m_{v-1}j_{l_{v-1}}},
	$$ 
	where $T_{m_s{l_s}}$ is in the row $i_{s+1}$, and $l_s$ are not necessarily distinct. Then without loss of generality we may assume that other term of $\alpha$ is
	$$
	x_{i_v}T_{\frac{x_{i_1}}{x_{i_2}}m_1j_{l_1}}T_{\frac{x_{i_2}}{x_{i_3}}m_2j_{l_2}}\dots T_{\frac{x_{i_{v-1}}}{x_{i_v}}m_{v-1}j_{l_{v-1}}}.
	$$
	Then the cycle starts from 1 going to $x_{i_1}$. But $T_{m_1j_{l_1}}$ is in the row $i_2$. Hence in the bipartite graph, $x_{i_1}$ is attached to $T_{\frac{m_1}{x_{i_2}}j_{l_1}}$. This edge creates $T_{\frac{x_{i_1}}{x_{i_2}}m_1j_{l_1}}$. The next edge in the cycle is going from $T_{\frac{m_1}{x_{i_2}}j_{l_1}}$ to $x_{i_2}$. This edge creates $T_{m_1j_{l_1}}$. Actually, vertices in one partition of the bipartite graph are $x_{i_1},\dots,x_{i_v}$, and in other partition vertices are $1,T_{\frac{m_1}{x_{i_2}}j_{l_1}},\dots,T_{\frac{m_{v-1}}{x_{i_v}}j_{l_{v-1}}}$. The vertex $T_{\frac{m_{s-1}}{x_{i_s}}j_{l_{s-1}}}$ is attached to vertices $x_{i_{s-1}}$ and $x_{i_s}$. These edges create $T_{\frac{x_{i_{s-1}}}{x_{i_s}}m_{s-1}j_{l_{s-1}}}$ and $T_{m_{s-1}j_{l_{s-1}}}$. Also, for $1<s<v$, $x_{i_s}$ is attached to $T_{\frac{{m_{s-1}}}{x_{i_s}}j_{l_{s-1}}}$ and $T_{\frac{m_s}{x_{i_{s+1}}}j_{l_s}}$. Finally, $x_{i_v}$ is attached to 1 and $T_{\frac{m_{v-1}}{x_{i_v}}j_{l_{v-1}}}$. From the argument we see that every such a cycle gives us a binary quasi-minor.
\end{discussion}

\begin{remark}
	We see that the directed cycles $C_M$ and $C_N$ in the proof of Theorem~\ref{main-result} (which have the same vertices) correspond to two family of bipartite graphs as argued in Discussion~\ref{cycle}. Each pair of these graphs which correspond to $C_M$ and $C_N$ have different vertices.
\end{remark}

\section{Concluding remarks and questions}\label{sec:Conclusion}

We consider the polynomial ring $k[\x]$. We also consider the poset $L=\{x_1,\dots,x_n\}$. Let $L_i$ ($1\le i\le r$) be subposets of $L$. Suppose $M_1,\ldots,M_r$ be a collection of monomials in $k[\x]$. We take ideals $I_i=L_i$-$\Borel(M_i)$, which means these ideals are generated by monomials which are obtained by Borel moves starting from $M_i$ on variables that are in $L_i$. We consider the multi-Rees algebra $k[\x][I_1t_1,\dots,I_rt_r]$. Similar to previous section we define $\bT$ variables and the polynomial ring $k[\x,\bT]$. Also, we define similarly the map $\phi$. We define a bipartite graph, where one partition of vertices are $t_1,\dots,t_r$ and the other partition of vertices are $x_1,\dots,x_n$. Edges are made by connecting variables in the posets $L_i$ to $t_i$. In~\cite{DJ20}, we have showed that if this bipartite graph is chordal (that means every cycle of length greater or equal to 6 has a chord), then there is a Gr\"{o}bner basis of quadrics with lexicographic order for $\ker(\phi)$. We pose the following question for interested reader:

\begin{question}
	If the bipartite incidence graph is non-chordal, how do the equations look like? Let consider the quasi-matrix $C_{\boldsymbol{a}}$. Since we deal with Borel moves for some $T_{mt_j}$, $m$ may not be in $I_j$. We delete these $T$ variables from $C_{\boldsymbol{a}}$ and we obtain another quasi-matrix, say, $A$. If we have monomial order as given in Convention~\ref{conv:MonomialOrder}, then is the Gr\"{o}bner basis of $\ker(\phi)$ formed by binary quasi-minors of $A$?
\end{question}  

Also, in~\cite{DJ20}, we have showed that when the associate incidence bipartite graph is chordal, then the multi-Rees algebra is Koszul. We have also posed this question in~\cite{DJ20}, that is this a necessary condition for the multi-Rees algebra $k[\x][I_1t_1,\dots,I_rt_r]$ to be Koszul? Here, in the case which is discussed in the present paper we show that this is a necessary condition.

\begin{proposition}\label{pro-koszul}
	Suppose that ideals $J_i$ are generated by subsets of $x_1,\dots,x_n$. Let $I_i=J_i^{a_i}$. Suppose $\calI=\{I_1,\dots,I_r\}$. If the incidence bipartite graph associated to the multi-Rees algebra is non-chordal, then the multi-Rees algebra $k[\x][\calI\bt]$ is not Koszul.
\end{proposition}
\begin{proof}
	Without loss of generality we may assume that we have a cycle shown in Figure~\ref{koszul}, of length $\ge 6$, which does not have a chord.
	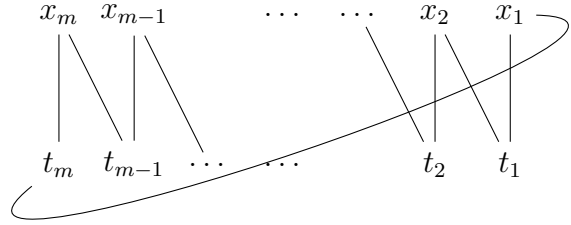
\begin{figure}
		\centering
		\begin{tikzpicture}
		\begin{scope}
		\node (x1)    at  (6,2) {$x_1$};
		\node (x2)    at  (5,2) {$x_2$};
		\node (a)    at  (4,2) {$\dots$};
		\node ()      at  (3,2) {$\dots$};
		\node (xm-1)      at  (1,2) {$x_{m-1}$};
		\node (xm)      at  (0,2) {$x_m$};
		
		\node (t1)    at  (6,0) {$t_1$};
		\node (t2)    at  (5,0) {$t_2$};
		\node ()      at  (3,0) {$\dots$};
		\node (b)      at  (2,0) {$\dots$};
		\node (tm-1)  at  (1,0) {$t_{m-1}$};
		\node (tm)  at  (0,0) {$t_m$};
		
		\end{scope}    
		
		\begin{scope}
		\path[-] (x1) edge (t1);
		\path[-] (t1) edge (x2);
		\path[-] (x2) edge (t2);
		\path[-] (t2) edge (a);
		\path[-] (xm) edge (tm-1);
		\path[-] (xm) edge (tm);
		\path[-] (tm-1) edge (xm-1);
		\path[-] (xm-1) edge (b);
		\path[-] (tm) edge[bend right=160] (x1);
		\end{scope}
		\end{tikzpicture}
		\caption{The cycle in Proposition~\ref{pro-koszul}}
		\label{koszul}
	\end{figure}
	Then we see that the binary quasi-minor 
	$$
	\alpha=T_{x_1^{a_1}t_1}T_{x_2^{a_2}t_2}\dots T_{x_m^{a_m}t_m}-T_{x_1^{a_1-1}x_2t_1}T_{x_2^{a_2-1}x_3t_2}\dots T_{x_m^{a_m-1}x_1t_m},
	$$
	cannot be generated by any quadrics. Because if $\alpha=\sum \beta_i\gamma_i$, where $\gamma_i$ are quadrics, then there is at least one $\gamma_j$ such that two factors of the first term of $\alpha$ is one term of $\gamma_j$. Then $\phi(\gamma)\neq 0$, which is a contradiction. 
\end{proof}

We have a similar argument for the multi-fiber ring.

\begin{proposition}
	Suppose that ideals $J_i$ are generated by subsets of $x_1,\dots,x_n$. Let $I_i=J_i^{a_i}$. Suppose $\calI=\{I_1,\dots,I_r\}$. If the incidence bipartite graph associated to the multi-fiber ring $k[\calI\bt]$ is non-chordal, then the multi-fiber ring is not Koszul.
\end{proposition}

\section{Acknowledgments}
	I would like to express my gratitude to my former advisor, Mark Johnson, for his valuable comments. I would also like to express my gratitude to Michael DiPasquale, as he gave many valuable comments to make this paper better.


\begin{bibdiv}
	\begin{biblist}
		\bib{CLS19}{article}{
			title={Multi-Rees algebras and toric dynamical systems},
			author={Cox, David A.}
			author={Lin, Kuei-Nuan},
			author={Sosa, Gabriel},
			journal={Proceedings of the American Mathematical Society},
			volume={147},
			year={2019},
			number={11},
			pages={4605--4616}
		}
		\bib{BC17}{article}{
			title={Linear resolutions of powers and products},
			author={Bruns, Winfried},
			author={Conca, Aldo},
			journal={Singularities and computer algebra},
			pages={47--69},
			year={2017},
			publisher={Springer}
		}
		\bib{BC17b}{article}{
			title={Products of {B}orel fixed ideals of maximal minors},
			author={Bruns, Winfried},
			author={Conca, Aldo},
			journal={Advances in Applied Mathematics},
			volume={91},
			year={2017},
			pages={1--23}
		}
		\bib{DST95}{article}{
			title={Gr\"{o}bner bases and triangulations of the second hypersimplex},
			author={De Loera, Jes{\'u}s A.}
			author={Sturmfels, Bernd},
			author={Thomas, Rekha R.},
			journal={Combinatorica},
			volume={15},
			number={3},
			pages={409--424},
			year={1995}
		}
	   \bib{DJ20}{article}{
	title= {Koszul multi-Rees algebras of principal $L$-Borel Ideals},
	author={DiPasquale, Michael},
	author={Jabbar Nezhad, Babak},
	journal={Journal of Algebra},
	volume={581},
	pages={353--385},
	year={2021}
}
		\bib{eisenbud2003rees}{article}{
			title={What is the Rees algebra of a module?},
			author={Eisenbud, David},
			author={Huneke, Craig},
			author={Ulrich, Bernd},
			journal={Proceedings of the American Mathematical Society},
			volume={131},
			number={3},
			pages={701--708},
			year={2003}
		}
		\bib{HHO18}{book}{
			title={Binomial ideals},
			author={Herzog, J{\"u}rgen},
			author={Hibi, Takayuki},
			author={Ohsugi, Hidefumi},
			volume={279},
			year={2018},
			publisher={Springer}
		}
		\bib{HHV05}{article}{
			title={Ideals of fiber type and polymatroids},
			author={Herzog, J{\"u}rgen},
			author={Hibi, Takayuki},
			author={Vladoiu, Marius},
			journal={Osaka Journal of Mathematics},
			volume={42},
			number={4},
			pages={807--829},
			year={2005}
		}
		\bib{hochster1972rings}{article}{
			title={Rings of invariants of tori, Cohen-Macaulay rings generated by monomials, and polytopes},
			author={Hochster, Melvin},
			journal={Annals of Mathematics},
			volume={96},
			pages={318--337},
			year={1972}
		}
		\bib{jabarnejad2016rees}{article}{
			title={Equations defining the multi-Rees algebras of powers of an ideal},
			author={Jabarnejad, Babak},
			journal={Journal of Pure and Applied Algebra},
			volume={222},
			pages={1906--1910},
			year={2018}
		}
		\bib{kustinpoliniulrich2017blowup}{article}{
			title={The equations defining blowup algebras of height three Gorenstein ideals},
			author={Kustin, Andrew R.},
			author={Polini, Claudia},
			author={Ulrich, Bernd},
			journal={Algebra and Number Theory},
			volume={11},
			number={7},
			pages={1489--1525},
			year={2017}
		}
		\bib{LP14}{article}{
			title={Rees algebras of truncations of complete intersections},
			author={Lin, Kuei-Nuan},
			author={Polini, Claudia},
			journal={Journal of Algebra},
			volume={410},
			year={2014},
			pages={36--52}
		}
		\bib{morey1996rees}{article}{
			title={Equations of blowups of ideals of codimension two and three},
			author={Morey, Susan},
			journal={Journal of Pure and Applied Algebra},
			volume={109},
			pages={197--211},
			year={1996}
		}
		\bib{moreyulrich1996rees}{article}{
			title={Rees algebras of ideals with low codimension},
			author={Morey, Susan},
			author={Ulrich, Bernd},
			journal={Proc. Amer. Math. Soc.},
			volume={124},
			pages={3653--3661},
			year={1996}
		}
		\bib{OHT13}{article}{
			title={Graver basis for an undirected graph and its application to testing the beta model of random graphs},
			author={Ogawa, Mitsunori},
			author={Hara, Hisayuki},
			author={Takemura, Akimichi},
			journal={Annals of the Institute of Statistical Mathematics},
			volume={65},
			number={1},
			pages={191--212},
			year={2013}
		}
		\bib{OH98}{article}{
			title={Normal polytopes arising from finite graphs},
			author={Ohsugi, Hidefumi},
			author={Hibi, Takayuki},
			journal={Journal of Algebra},
			volume={207},
			number={2},
			pages={409--426},
			year={1998}
		}
		\bib{OH99}{article}{
			title={Toric ideals generated by quadratic binomials},
			author={Ohsugi, Hidefumi},
			author={Hibi, Takayuki},
			journal={Journal of Algebra},
			volume={218},
			number={2},
			pages={509--527},
			year={1999}
		}
		\bib{RTT12}{article}{
			title={Minimal generators of toric ideals of graphs},
			author={Reyes, Enrique},
			author={Tatakis, Christos},
			author={Thoma, Apostolos},
			journal={Advances in Applied Mathematics},
			volume={48},
			number={1},
			pages={64--78},
			year={2012}
		}
		\bib{ribbe1999defining}{article}{
			title={On the defining equations of multi-graded rings},
			author={Ribbe, J.},
			journal={Communications in Algebra},
			volume={27},
			number={3},
			pages={1393--1402},
			year={1999}
		}
		\bib{simis2003rees}{article}{
			title={Rees algebras of modules},
			author={Simis, Aron},
			author={Ulrich, Bernd},
			author={Vasconcelos, Wolmer V.},
			journal={Proceedings of the London Mathematical Society},
			volume={87},
			number={3},
			pages={610--646},
			year={2003}
		}
		\bib{SVV98}{article}{
			title={The integral closure of subrings associated to graphs},
			author={Simis, Aron},
			author={Vasconcelos, Wolmer V.},
			author={Villarreal, Rafael H.},
			journal={Journal of Algebra},
			volume={199},
			number={1},
			pages={281--289},
			year={1998}
		}
		\bib{Sosa14}{article}{
			title= {On the Koszulness of multi-Rees algebras of certain strongly stable ideals},
			author={Sosa, Gabriel},
			journal={arXiv:1406.2188},
			year={2014}
		}
		\bib{sturmfels1996grobner}{book}{
			title={Gr{\"o}bner bases and convex polytopes},
			author={Sturmfels, Bernd},
			series={University Lecture Series},
			volume={8},
			publisher={American Mathematical Society},
			year={1996}
		} 
		\bib{vasconcelosulrich1993rees}{article}{
			title={The equations of Rees algebras of ideals with linear presentation},
			author={Ulrich, Bernd},
			author={ and Vasconcelos, Wolmer V.},
			journal={Math. Z.},
			volume={214},
			pages={79--92},
			year={1993}
		}
		\bib{vasconcelos1991rees}{article}{
			title={On the equations of Rees algebras},
			author={Vasconcelos, Wolmer V.},
			journal={J. Reine Angew. Math.},
			volume={418},
			pages={189--218},
			year={1991}
		}
		\bib{V95}{article}{
			title={Rees algebras of edge ideals},
			author={Villarreal, Rafael H.},
			journal={Communications in Algebra},
			volume={23},
			number={9},
			pages={3513--3524},
			year={1995}
		}
	\end{biblist}
\end{bibdiv}
\end{document}